\documentclass[12pt,letterpaper]{amsart}
\usepackage{graphicx}
\usepackage{amssymb,amscd,amsthm,amsxtra}
\usepackage{latexsym}
\usepackage{epsfig}
\usepackage{amsmath}
\usepackage{bm}
\usepackage{amssymb}
\usepackage{amsthm}
\usepackage[usenames,dvipsnames,svgnames,table]{xcolor}
\usepackage[linktocpage=true,colorlinks=true,linkcolor=Blue,citecolor=BrickRed,urlcolor=RoyalBlue]{hyperref}
\usepackage[alphabetic]{amsrefs}
\usepackage{comment}
\usepackage[colorinlistoftodos,prependcaption,textsize=tiny]{todonotes}
\usepackage{caption}
\usepackage{amscd}
\usepackage[shortlabels]{enumitem}
\usepackage[mathscr]{euscript} 

\usepackage{mathrsfs} 
\usepackage{pifont}

\newlist{thmlist}{enumerate}{1}
\setlist[thmlist]{label=(\alph{thmlisti}), ref=\thethm.(\alph{thmlisti}),noitemsep}

\usepackage{thmtools}

\usepackage{nameref,hyperref}
\usepackage[capitalize]{cleveref}

\Crefname{thm}{Theorem}{Theorems}
\Crefname{lem}{Lemma}{Lemmas}

\makeatletter
\newcommand{\mylabel}[2]{#2\def\@currentlabel{#2}\label{#1}}
\makeatother

\numberwithin{equation}{section} 

\newtheorem{theorem}{Theorem}

\newtheorem{lemma}{Lemma}
\newtheorem{corollary}{Corollary}
\newtheorem{proposition}{Proposition}

\theoremstyle{remark}
\newtheorem{remark}{Remark}[section]

\def \suchthat {\ \big | \ }

\title[Schauder-type estimates for degenerate equations]{Schauder-type estimates for  fully nonlinear degenerate elliptic equations}
\author[T. M. Nascimento]{Thialita M. Nascimento}
\address{Department of Mathematics, Iowa State University, 396 Carver Hall, 50011, Ames, IA, USA}{}
\email{thnasc@iastate.edu}

\begin{document}
\maketitle

\date{} 

\begin{abstract} 
In this paper, we examine regularity estimates for solutions to fully nonlinear,  degenerated elliptic equations, at interior vanishing source points. At these points, we obtain Schauder-type regularity estimates, which depend on the H\"older-like source-ellipticity vanishing rate. 

\tableofcontents

\end{abstract}

\section{Introduction}\label{sec1}

Regularity estimates for degenerate fully nonlinear elliptic equations of the form 
\begin{equation}\label{model fully nonl eq}
        |\nabla u|^{\gamma} F(D^2 u) = f(x)
\end{equation}
have been subject to extensive study since the works of Birindelli and Demengel in a series of papers \cites{BD1, BD2, BD3, BD4, BD5} . It is  quite easy to see why such equations attract so much attention. A large number of physical and social phenomena that involve diffusive processes are modeled by elliptic operators whose ellipticity degenerates  along an {\it a priori} unknown set that depends on the solution itself. From the mathematical point of view,  one of the challenges in the analysis of models like \eqref{model fully nonl eq} lies in the  fact that along the set of critical points, namely  $\mathcal{C}(u) = \{ |\nabla u | = 0 \}$,  one loses smoothness features that come from diffusion efficiency and thus regularity theory for such equations require the development of more sophisticated mathematical tools. 

 The  difficulty of establishing regularity estimates for solutions to these equations along their singular sets is very similar, in nature and analytically,  to free boundary problems.  Indeed, the set of critical points $ \mathcal{C}(u)$ {\it is}  the (non-physical) free boundary of the problem. Moreover, to make an even stronger parallel (see for instance \cite{ART}), equations of the form \eqref{model fully nonl eq} are closely related to a rather general class of  free boundary problems studied in \cite{Caff-Sal}(where the right-hand side is of the form $\chi_{\{ \nabla u \neq 0 \}}$). Therefore, many of the techniques from free boundary problems can be imported to obtain regularity estimates for equations like $\eqref{model fully nonl eq}$. In this regard, the regularity theory of degenerate elliptic equations, only in the past few years, has seen an impressive progress , see e.g, \cite{ART, DFQ, IS} and references therein.

For instance, in the groundbreaking work \cite{IS},  C. Imbert and L. Silvestre proved that solutions to 
$$
    |\nabla u|^{\gamma} F(D^2 u) = f(x)
$$
with $f \in L^{\infty} (B_1)$, are $C_{loc}^{1,\alpha}$, for some $0< \alpha \le \frac{1}{1+\gamma}$.   This regularity is optimal, even if the source term $f$ is H\"older continuous and $F = \Delta$. However, simply saying that solutions are $C^{1,\alpha}$ does not tell the whole story. If fact,  borrowing the example presented  in \cite{IS}, for any $\alpha \in (0,1)$, the function $u(x) = |x|^{1 + \alpha}$ satisfies
$$
    | \nabla u|^{\gamma} \Delta u = C(n, \alpha) |x|^{(1 + \alpha)(1 + \gamma) - (2 + \gamma)}.
$$
 Writing $\theta = (1 + \alpha)(1 + \gamma) - (2 + \gamma)$, one can see that  such exponent $\theta$ is smaller or equal  then the degeneracy rate, i.e. $\gamma.$ Moreover, $u$ is $C^{1 \frac{1 + \theta}{1 + \gamma}}$ at the origin (a critical point for $u$). 
 
 Note, however, that for any given  $\theta > \gamma$,  the function  
\begin{equation}\label{example of hessian cont}
    u(x) = |x_{n}|^{2 + \frac{\theta - \gamma}{1+\gamma}}
\end{equation} is exactly $C^{2, \frac{\theta - \gamma}{1 + \gamma}}$ at the origin and satisfies 
$$
    |\nabla u|^{\gamma} \Delta u = f(x)
$$
where 
$$
    f(x)=C |x_n|^{\gamma + \frac{\gamma(\theta - \gamma)}{1 + \gamma} + \frac{\theta - \gamma}{1 + \gamma}}   =  C |x_n|^{\theta}  \le C|x |^{\theta} .
$$

These examples show that the H\"older continuity of the source function directly influences, {\it in a quantitative way},  the regularity estimates for solutions to degenerate equations. Moreover, they indicate that, at least at certain meaningful points (to be specified later),  we can indeed surpass the $C^{1,\alpha}$ regularity of $u$  if the H\"older exponent of the right-hand side is larger then the growth rate of the gradient term in the equation. In other words, if the source term vanishes faster than ellipticity degenerates. 

In many cases, one is required to obtain fine control on how fast solutions separate from their tangent planes, or how fast it grows away from, say, its zero level set. For instance, such questions are naturally raised in problems of free boundary type. Therefore, knowing the optimal regularity available for solutions is essential. Inspired  by the above examples, in this article we explore the {\it quantitative} influence of the H\"older-like vanishing rate of the source term in the regularity estimates for degenerate elliptic equations of the form 
\begin{equation}\label{main eq intr intr}
    \mathcal{D}(x, \nabla u) F(x, D^2 u) = f(x) \quad \text{in} \quad B_1 \subset \mathbb{R}^n
\end{equation}
where $\mathcal{D}$ is the degeneracy law of the model ($\mathcal{D}(x, \nabla u) \sim |\nabla u|^{\gamma}$ ), $F$ is responsible for the diffusion, and $f$ is a H\"older continuous function.  At an interior vanishing source point, we obtain sharp, higher-order  regularity estimates,  which depend, in addition, on the vanishing rate of the source function $f$ (or the H\"older exponent of $f$, if the readers prefer). 

To better highlight the importance of such regularity improvement, we make a  parallel with the Schauder regularity theory for  elliptic equations.  We recall that Schauder theory asserts that solutions to an elliptic equation, with minimal regularity,  must be as smooth as data permits. For instance,  a viscosity solution to a linear elliptic equation
$$
    a_{ij}(x) \partial_{ij} u = f(x)
$$
with $ 0 < \lambda \le a_{ij}(x) \le \Lambda$, and data $a_{ij} (x) , f \in C^{0,\theta}$, is locally of class $C^{2, \theta}$. Moreover, there exists a constant $C >0$ depending only on ellipticity constants $(\lambda, \Lambda)$, dimension and the $\theta$-H\"older continuity of the data, such that 
$$
    \| u \|_{C^{2,\theta}(B_{1/2})} \le C \| u\|_{L^{\infty}(B_1)} .
$$

Launched by Schauder in the 1930's, this estimate has a vast range of applications in the theory of PDEs and it is considered to be of one the finest treasure in regularity theory. For a comprehensive reading, the curious reader is referred to the classical books \cite[Chapter 6]{GT}, \cite[Chapter 6]{Morrey} and the original paper by J. Schauder \cite{Schauder} .  

 The counterpart of the Schauder theory in the fully nonlinear setting, obtained by Luis Caffarelli  in \cite{Caff, Caff2}, asserts that  under $C^{0,\bar{\alpha}}$ continuity of the coefficients and of the source term $f$, together with $C^{2, \alpha_{F}}$ {\it a priori} estimates, for the constant coefficient model $F(0, D^2 h) = 0$, with $0 < \bar{\alpha} < \alpha_{F}$, one is able to show that solutions to the variable coefficient equation, $F(x, D^2 u) = f(x)$ are $C_{loc}^{2,\bar{\alpha}}$. 

Thus looking at the degenerate nonlinear elliptic model
\begin{equation}\label{degenerate model}
    |\nabla h(x)|^{\gamma} F(x, D^2 h) = f(x),
\end{equation}
and interpreting the term $(|\nabla h(x)|^{\gamma} )$ as if it were part of the coefficients of the correspondent diffusion, one expects that Caffareli's  analysis  applies for such equations, at least along non-degenerate points.  In fact, the heuristics of our results go like this: Away from their set of critical points $\mathcal{C}(h)$, solutions are as regular as solutions to a variable coefficient equation $G(x, D^2 u) = g(x)$. That is, if we can quantify how far a point is from $\mathcal{C}(h)$, then after a proper re-scaling, the equation becomes a uniformly elliptic equation in a domain of fixed sized.  As we approach the set of critical points, $h$ loses its smoothness. However, at a in interior vanishing source point $x_0$, by profiting from a sort cancellation effect (\cite[Lemma 6]{IS}), we may understand equation $\eqref{degenerate model}$ as
\begin{equation}\label{eq for analogy}
    \tilde{F}(x_0, D^2 u(x_0)) = 0 .
\end{equation}
Therefore,  we are able to approximate solutions to \eqref{degenerate model} by solutions to \eqref{eq for analogy} in a small neighborhood of $x_0$, and hence improve the regularity of solutions at such points up to an interpolation between the vanishing rates of $f$ and of the ellipticity. The above discussion describes the dichotomy used to prove our first main result, Theorem \ref{main result}. 

 The rest of the paper is organized as follows: In section \ref{sct Prelim},  we outline the main assumptions on the structure of the equation,   formally state and comment the main results, and  provide insights into their applications.  Section \ref{sct improved gradient reg} is dedicated to proving Theorem \ref{main result} and Theorem \ref{C^{1,1-} reg}.  For this purpose, we follow the strategy of locating the source-vanishing points within $B_{1/2}$ and applying the dichotomy described in the last paragraph. Such strategy relies on imported techniques from the theory of geometric free boundary problems, as those utilized in  \cite{Teix3, Teix2, Teix1}, as well as in \cite{AS}.   In section  \ref{sct hessian cont}, we prove Hessian continuity at local extrema points  that also vanish the source term. Namely, Theorem \ref{hessian cont}. At a local extrema, even though the gradient vanishes, the solution still behaves nicely enough as so to be interpreted as a solution to an uniformly elliptic equation, whose coefficient vanishes at such points.  Thus, if  $\theta > \gamma$  we obtain $C^{2,\varepsilon}$ regularity at these points, for an $\varepsilon \in (0,1)$ depending on the  difference between the vanishing rates of the source term and degeneracy. 
\section{Assumptions, main results and further insights} \label{sct Prelim}

We begin this section by defining the terminology and notation used throughout the paper and setting the proper mathematical set up in which the main results will be stated and proven. 

We will denote by $Sym (n)$ the  space of all $n \times n$ symmetric matrices in $\mathbb{R}^n$ . Let  $F: B_1 \times Sym (n) \to \mathbb{R}$ be a fully nonlinear, uniformly elliptic operator. That is,  given  $x \in B_1$ and $M, N \in Sym(n)$, there holds
\begin{equation}\label{unif ellipticity}  \tag{A1}
    \mathcal{M}_{\lambda,\Lambda}^-(M-N) \leq F(x,M) - F(x,N) \leq \mathcal{M}_{\lambda,\Lambda}^+(M-N),
\end{equation}
for constants $0 < \lambda \leq \Lambda$. We recall that $\mathcal{M}^+_{\lambda,\Lambda}$ and $\mathcal{M}^-_{\lambda,\Lambda}$ stand for the {\it Pucci Extremal Operators} defined as
$$
    \mathcal{M}_{\lambda,\Lambda}^{-} (M)  =  \inf \{Tr(AM) \colon \lambda I\leq A \leq \Lambda I  \},
$$
and $\mathcal{M}^+_{\lambda,\Lambda}(M) = - \mathcal{M}^-_{\lambda,\Lambda}(-M)$. 

We are interested in equations of the form 
\begin{equation}\label{main eq intr}
    \mathcal{D}(x, \nabla u) F(x, D^2 u) = f(x) \quad \text{in} \quad B_1 \subset \mathbb{R}^n .
\end{equation}
where $f \in L^{\infty}(B_1)$, $F$ satisfies \eqref{unif ellipticity} and the degeneracy law, given by $\mathcal{D}: B_1 \times \mathbb{R}^n \to \mathbb{R}$,  satisfy 
    \begin{equation}\label{degener cond} \tag{A2}
         \tilde{\lambda} | \vec{b} |^{\gamma} \le \mathcal{D}(x, \vec{b} ) \le \tilde{\Lambda} | \vec{b} |^{\gamma}
    \end{equation}
for some $0 < \tilde{\lambda} \le \tilde{\Lambda}$ and $\gamma \in (0,1)$. 
An example of such a function is 
$$
    \mathcal{D}(x, \vec{b} ) = q(x) |\vec{b}|^{\gamma}
$$
for some positive, bounded function $\tilde{\lambda} \le q(x) \le \tilde{\Lambda}.$

Hereafter, any operator F satisfying the ellipticity assumption \eqref{unif ellipticity} will be referred  in this paper as a $(\lambda, \Lambda)$-elliptic operator. Also, for normalization purposes, and without loss of generality, we will assume that $F(x, 0) = 0$ for all $x \in B_1.$

We will denote the oscillation of the coefficients of the operator $F$, by
$$
    \omega_F(y,x) = \omega_F(y, x) = \sup\limits_{M \in Sym(n)\setminus\{0\}} \frac{|F(x,M) - F(y,M)|}{(\|M\| + 1)} ,
$$
and 
$$
    \omega_{F}(x) = \omega_{F}( 0, x). 
$$
When there is no ambiguity we will denote the oscillation of $F$ simply by $\omega$. To be in accordance to \cite{Caff} and \cite{Teix1}, we will assume H\"older continuity on the coefficients. That is, we assume that there exist  constants $K_1 > 0$ and $\bar{\alpha} \in (0,1)$ such that 
\begin{equation}\label{contin of coeff}\tag{A3}
   \omega_F(x,y)  \leq K_1 |x - y|^{\bar{\alpha}}
\end{equation}
for all $M \in Sym(n)$ and all $x,y \in B_1.$

Since our analysis will revolve around source-vanishing points, hereafter we  assume that $f$ is $C^{0,\theta}$ at $0$, and $f(0) = 0$.  That is,  we assume that for some $K_2 > 0$ and $\theta \in (0,1)$, there holds
\begin{equation}\label{holder cont of f}\tag{A4}
    |f(x)| \le K_2 |x|^{\theta} .
\end{equation}

For an operator $G: B_1 \times \mathbb{R}^n \times Sym(n) \to \mathbb{R}$, solutions to $G(x, \nabla u, D^2 u) = 0$ are understood in sense of  \cite{CIL92}. See also \cite{Caff-Cabre}.  We recall that, see e.g. \cite{Caff}, \cite{Teix1}, that under such continuity condition on the coefficients, viscosity solutions  to
$$
    F(x,D^2 u)= f(x) \in L^{\infty} (B_1)
$$
 are locally of class $C^{1, \beta}$ , for any $0 < \beta < \alpha_{0}$ , where $\alpha_{0}$ is the optimal H\"older exponent coming from  the Krylov-Safonov $C^{1,\alpha}$ universal regularity theory, for solutions to constant coefficient, homogeneous equation $F(D^2 h) = 0$.

We are now in position to state our main results. Before continuing, let us declare that any constant that might depend only upon ellipticity constants $(\lambda, \Lambda)$, dimension, $\| u\|_{\infty}$,  H\"older continuity of the coefficients and of $f$,  and the regularity  estimates of the constant coefficient equation $F(D^2 h) =0$, will be called {\it universal}. 

The first is an improved gradient regularity estimate. 
\begin{theorem}\label{main result}
    Let $u \in C(B_1)$ be a bounded viscosity solution to \eqref{main eq intr}. Assume that $F:B_1 \times Sym(n) \to \mathbb{R}$ satisfy \eqref{unif ellipticity},  \eqref{contin of coeff}, and  that $f $ satisfies \eqref{holder cont of f}. Then  $u$ is of class $ \displaystyle C^{1, \min\{\alpha_0^{-}, \frac{1+\theta}{1 + \gamma}\} }$ at the origin.  That is, for any 
       $$
        \beta \in (0, \alpha_0) \cap \left( 0, \frac{1 + \theta}{1 + \gamma} \right] 
       $$ we have , 
        $$
            | u(x) - u(0)  - \nabla u(0) \cdot x| \le  C_{\beta}  |x |^{1 + \beta}.
        $$
for all $x \in B_{1/4}(0)$, where $C_{\beta} > 0$ depending only on $\beta$, and universal parameters.
\end{theorem}
An immediate consequence of Theorem \ref{main result} is the following:
\begin{corollary}
    Under the hypothesis of Theorem \ref{main result},  if $\theta > \gamma$ then $u$ is of class $\displaystyle C^{1, \alpha_0^{-}}$ at $x_0 = 0$. That is, for any $ \beta \in (0, \alpha_0)$,
        $$
            | u(x) - u(0)  | \le  C_{\beta}|x |^{1 + \beta}.
        $$  
for all $x \in B_{1/4}(0)$, where $C_{\beta} > 0$ depending only on $\beta$, and universal parameters. 
\end{corollary}

When more structure on the diffusion operator $F$ is assumed,  a careful analysis of the proof of Theorem \ref{main result}, revels that, at vanishing source points,   $C^{1,1^{-}}$ regularity holds true provided the coefficients are "continuous enough" as to allow the use of a priori $C^{1,1}$ estimates for $F$-harmonic functions. That is, if we  further assume  that for any $N \in Sym(n)$ with $F(0,N) = 0$, solutions to $F(0, D^2 h + N) = 0$ satisfy
  \begin{equation}\label{a priori hessian est}\tag{A5}
      \| h\|_{C^{1,1}(B_t)} \le \bar{C} r^{-2 }\|h\|_{L^{\infty}(B_1)}.
  \end{equation}
 for some universal constant $\bar{C} >0$,  we can summarize the above discussion as the following theorem:
 
 \begin{theorem}\label{C^{1,1-} reg}
     Let $u \in C(B_1)$ be a viscosity solution to \eqref{main eq intr}. In addition to the hypothesis of Theorem \ref{main result}, further assume that \eqref{a priori hessian est} is in force, and that $\theta \ge \gamma$.  Then, $u$ is of class $C^{1, 1^{-}}$ at the origin. 
        That is, for any $ \beta \in (0, 1)$,
        $$
            | u(x) - u(0)  - \nabla u(0) \cdot x| \le  C_{\beta}|x |^{1 + \beta}.
        $$ 
for all $x \in B_{1/4}(0)$, where $C_{\beta} > 0$ depending only on $\beta$, and universal parameters.
 \end{theorem}

 The final result  provided in this work is Hessian continuity at local extrema points. As suspected by example \eqref{example of hessian cont}, at an extrema point (say, a minimum point) that vanishes the source term at a faster rate than ellipticity degenerates, solutions are twice continuously differentiable. More precisely we have 
 
\begin{theorem}\label{hessian cont}
    Under the hypothesis of Theorem \ref{main result}. Assume in addition that $0$ is a local extrema for $u$. Then, $u$ is twice differentiable at the origin, and 
    $$
        | u(x) - u(0) | \le C |x|^{2 + \varepsilon}
    $$
    for any $ x \in B_{1/8}(0)$, where 
    \begin{equation}\label{varepsilon}
        \varepsilon:= \frac{\theta - \gamma}{1+ \gamma}
    \end{equation}
    and $C > 0$ is a universal constant.  That is, $u$ is of class $C^{2, \frac{\theta - \gamma}{1 + \gamma}}$ at the origin, with $D u(0) = D^2 u (0) = 0$. 
\end{theorem}
To highlight the robustness of the above theorem, we note that no further assumption on $F$ is assumed. That is, no a priori hessian estimates are required.  On the other hand, if say $C^{2, \alpha_F}$ a priori estimates  are available for the constant coefficient, homogeneous model $F(0, D^2 h) = 0$, and  the vanishing rate of source term, $\theta$, is so that $\theta > \gamma$, it remains reasonable to anticipate Hessian continuity,  even when the vanishing-source point in not an extrema point. This, however, seems to be a rather delicate and important question, and the author wish to return to this issue in a future research. 


Insofar as applications of the above results is concerned, one  of particular interest comes from  the so called generalized eigenvalue problems. We briefly recall (see e.g., \cite{BD3, BD4}) that the numbers
$$
    \lambda^{+}: = \sup\left\{ \lambda \suchthat \exists \varphi \ge 0, |\nabla \varphi|^{\gamma} F( D^2 \varphi) +  \lambda (\varphi)^{1 + \gamma}  \le 0 \, \text{in}\, \Omega \right\}
$$
and 
$$
    \lambda^{-}: \sup\left\{ \lambda \suchthat \exists \psi \le 0, |\nabla \psi|^{\gamma} F( D^2 \psi) +  \lambda (\psi)^{1 + \gamma}  \ge 0 \, \text{in}\, \Omega \right\}
$$
are called {\it principle eigenvalues} if there exists a nontrivial solution to the Dirichlet problem
\begin{equation}\label{EigenVal}\tag{$E_{\lambda,\gamma}$}
   \left\{\begin{matrix}
 |\nabla u|^{\gamma} F( D^2 u) = \lambda^{\pm} |u|^{\gamma} u  &\text{in} & \Omega\\ 
 u = 0 &\text{on} & \partial \Omega  .
\end{matrix}\right.
\end{equation}

The continuity of the gradient is naturally raised (see e.g., \cite{BD4}) when one tries to verify the simplicity of the eigenvalues. 

Another, equally important question concerns the geometry of the nodal sets $\mathcal{Z}(u): = u^{-1} (\{0 \})$. Since bounded solutions to \eqref{EigenVal} are $C^{1,\alpha}$, then each  nodal set is divided into its regular part
$$
    \mathcal{R}(u) = \{ x \in \Omega : u(x) = 0 \, \text{and} \, |\nabla u(x)| \neq 0 \}
$$
and, its singular part
$$
    \mathcal{S}(u) = \{ x \in \Omega : u(x) =  |\nabla u(x)| = 0 \}.
$$
Since the regular part, $\mathcal{R}(u)$ is locally a $C^{1, \alpha}$ $(n-1)$-dimensional surface, the structure of the nodal sets are left to the study of its singular part. While the singular part may exhibit singularities like cusps, the results here show that, even for  the problem in heterogeneous media, the solutions itself are actually more regular at singular points.

Indeed, for instance, if $ 0 \in \mathcal{S} (u) $, then 
$$
    |u(x)| \le C |x|^{1 + \alpha_{*}}
$$
and hence, the right hand side of \eqref{EigenVal} satisfies
$$
    | (\lambda^{\pm} |u|^{\gamma} u ) (x) | \le |\lambda^{\pm} C|x|^{(1 + \gamma)(1 + \alpha_*)} \le\tilde{C} |x|^{1 + \gamma}.
$$
Therefore, the results obtained in this work assure higher regularity of solutions at the origin (a source vanishing point).

By considering, $\mathcal{Z}(u)$ as an abstract free boundary, this gain of smoothness is no surprise.  The very same phenomena occurs in classical obstacle-type problems, $\mathcal{L} u = f\chi_{\{ u > 0 \} }$. This was first observed by Weiss in \cite{Weiss}, where it was shown that solutions to the classical obstacle problem,  are actually {\it smoother} at singular free-boundary points, while the free boundary itself may develop cups at such points. 

More generally, one can consider problems of the form 
$$
    |\nabla u |^{\gamma} F(x, D^2 u) = g(x, u)
$$
with $|g(x, u)| \le C |u|^{m}$, for $m \in (0,2).$

The case $ m \in (0, 1)$ has been studied in many contexts. For instance, in \cite{LRS} and \cite{PQS}, the authors treat the case $m \in (0,1)$ as  one-phase free boundary problems. An advantage of the techniques used here is the gain of smoothness along subsets of the zero level set even for sign-changing solutions. Moreover, along the "non-physical" singular part $\mathcal{S} (u) = \{ x : u(x) = |\nabla u(x) | = 0 \} $, solutions are smoother in the sense that these points foster a higher degree of regularity for solutions. This is a remarkable result, since as ellipticity degenerates near this singular set, solutions are expected to be less regular at such points.

\section{Improved gradient regularity} \label{sct improved gradient reg}

This section is devoted to the proof of Theorem \ref{main result}.  We start off by  commenting on the scaling properties of the equation.   If $u \in C(B_1)$ is a viscosity solution to \eqref{main eq intr}, then for numbers $\kappa > 0$ and $r \in (0,1)$, the function $v:B_1 \to \mathbb{R}$ defined by $v(x) = \kappa u(rx)$, solves
 $$
    \tilde{\mathcal{D}}(x, \nabla v) \tilde{F} (x, D^2 v) = \tilde{f}(x)
 $$
 where 
 $$
    \tilde{F} (x, M) = \kappa r^2 F(rx , \kappa^{-1} r^{-2} M ),
$$ 
is a uniformly elliptic operator with same ellipticity constants as $F$. 
$$
    \tilde{\mathcal{D}}(x, \vec{b} ) = (\kappa r)\mathcal{D}(rx, (\kappa r)^{-1}\vec{b} )
$$ satisfies the degeneracy condition \eqref{degener cond}, with the same constants. And 
$$
    \tilde{f}(x) = r^{2 + \gamma} \kappa^{1 + \gamma} f(rx).
$$ Therefore, up to a normalization, i.e, by choosing $\kappa = (\| u \|_{\infty} + 1 )^{-1}$, $v$ satisfies an equation with the same structural assumptions as $u$.  Moreover, with this choice of $\kappa$ and since $r\in (0,1)$, we see that $\tilde{f}$ satisfies \eqref{holder cont of f} with the same constant. 

Therefore, in what follows, we will always assume, with no loss of generality, that a bounded viscosity solution $u$ to \eqref{main eq intr} is a normalized solution, i.e., $|u| \le 1$ . 

Also, since equation \eqref{main eq intr} has no explicit dependency on $u$, then up to a translation, hereafter we may define for a function $v \in C^1$, with no loss of generality,
$$
    \mathcal{C}(v): = \{ x \in B_1 \suchthat v(x) = |\nabla v(x) | = 0 \} .
$$

\subsection{Gradient Approximations}\label{sct approx}
To follow the dichotomy strategy discussed in the introduction session, we first  obtain a primary compactness result, which asserts that:  if, at $x_0 = 0$, the gradient is tiny (i.e., $0$ is sufficiently close to $\mathcal{C}(u)$), the coefficients are nearly constant, and the source function is nearly zero, then we can find an  $\mathcal{F}$-harmonic function close to $u$, for which $0$ is a critical point. 

\begin{lemma}[Approximation lemma]\label{Approx lemma}
Let $\vec{b} \in \mathbb{R}^n$ and $u \in C(B_1)$ be a normalized viscosity solution to 
\begin{equation}\label{main eq intr approx}
\mathcal{D}(x, \nabla u  ) F(x, D^2 u ) = f(x) \quad \text{in} \quad B_1.
\end{equation}
  Given $\delta > 0$ there exists $\eta = \varepsilon(\delta, n, \lambda, \Lambda) > 0$,  such that if 
\begin{equation}\label{small regime}
|D u(0)| + \| \omega_F \|_{L^{\infty}(B_1)} + \|f\|_{L^{\infty}(B_1)}< \eta
\end{equation}
then there exists a function $h :B_{3/4} \to \mathbb{R}$ and a $(\lambda, \Lambda)$-  elliptic, constant coefficient operator $\mathcal{F}$ such that 
\begin{equation}\label{eq of approx lemma}
     \mathcal{F} (D^2 h) = 0 \quad 
    \text{in}\quad  B_{3/4}, 
\end{equation} in the viscosity sense, $ 0 \in \mathcal{C}(h)$, and 
$$
\| u - h \|_{L^{\infty}(B_{1/2})} < \delta. 
$$
\end{lemma}

\begin{proof}
Suppose, by contradiction, that there exists $\delta_0 > 0$ and  sequences of functions  $\mathcal{D}_k : B_1 \times \mathbb{R}^n \to \mathbb{R}$, a sequence of $(\lambda, \Lambda)$-elliptic operators $F_k : B_1 \times Sym(n) \to \mathbb{R}$, a sequence of viscosity solutions $\{ u_k\}$, and a linked sequence of functions $\{f_k\},$ such that 

    \begin{equation}\tag{i}
         -1 \le u_k \le 1 ;
    \end{equation}

     \begin{equation}\tag{ii}
         \mathcal{D}_k (x, \nabla u_k )  F_{k} (x, D^2 u_k )  = f_k(x)  \quad \text{in} \quad  B_1;
     \end{equation}
      and 

    \begin{equation}\tag{iii}
        \displaystyle |D u_k(0)| + \| \omega_k \|_{L^{\infty}(B_1)}  + \|f_k\|_{L^{\infty}(B_1)}< \eta
    \end{equation}
where $\omega_k$ is the oscillation of the coefficients of the operators $F_k$. However, 
    \begin{equation}\label{Appr Lem contrad eq}
             \inf\limits_{h \in \mathbb{F}} \| u_k - h\|_{L^{\infty} (B_{1/2})} \ge \delta_0 > 0, 
    \end{equation} 
    where 
    $$
        \mathbb{F}: = \left\{ h \suchthat \text{ $\mathcal{F} (D^2 h ) = 0 $  for some   $\mathcal{F}$ verifying \eqref{unif ellipticity}  and $0 \in \mathcal{C}(h)$} \right\} .
    $$
Initially, we know from \cite{ART} that the sequence $\{ u_k \}_{k \in \mathbb{N} }$ is pre-compact in the $ C_{loc}^{1,+} (B_1)$-topology.  Thus, passing to a subsequence if necessary, $\{ u_k \}$ converges locally uniformly in $B_1$ to a continuous function  $u_{\infty}$ in the $C^{1, +}$ topology.  
By uniform ellipticity and (iii),  $F_k (x, M) \to F_{\infty}(M)$ locally uniformly in  $B_1 \times Sym(n)$ to some constant coefficient, $(\lambda, \Lambda)$-operator, and  $f_k \to 0$  uniformly in $B_1$. Moreover, using the degeneracy condition \eqref{degener cond} there holds
    $$
\left\{\begin{matrix}
\lambda | \nabla u_k  |^{\gamma} | F_k( x, D^2 u_k ) | \le  |f_{k} (x)| \\
                    \\
 \Lambda | \nabla u_k  |^{\gamma} | F_k( x, D^2 u_k ) | \ge  |f_{k} (x)| .
\end{matrix}\right.
$$
Hence, by stability of viscosity solutions, we obtain that $u_{\infty}$ solves, in the viscosity sense, the following problem 
$$
    |\nabla u_{\infty} |^{\gamma} F_{\infty} (D^2 u_{\infty}) = 0 \quad \text{in} \quad B_{3/4},
$$
and  
$$
    u_{\infty}(0) = |\nabla u_{\infty} (0)| = 0.
$$
By \cite[Lemma 6]{IS}, $u_{\infty}$ actually solves 
$$
    F_{\infty} (D^2 u_{\infty}) = 0 \quad \text{in} \quad B_{3/4}. 
$$

Therefore, $u_{\infty}$ belongs to the admissible functional set $\mathbb{F}$ and hence we get a contradiction with \eqref{Appr Lem contrad eq} for $k$ sufficiently large. 
\end{proof}

Next, we proceed to obtain sharp regularity at vanishing-source points, for which the gradient is small enough. 
\subsection{Regularity for small gradient }\label{crit pt reg}
 First, we  recall that  bounded viscosity solutions  to \eqref{main eq intr} are $C_{loc}^{1,+} (B_1)$, with universal estimates (see e.g.  \cite{ART} and \cite{IS} ). Next, fixing a number 
\begin{equation}\label{sharp exponent beta}
    \beta \in (0, \alpha_0) \cap \left( 0, \frac{1 + \theta}{1 + \gamma} \right],
\end{equation}
we aim to show that $u$ is  $C^{1, \beta}$ at the origin.  We also recall that we are assuming that $|u| \le 1$ and that  $ |f(x)| \le K_2 |x|^{\theta}$, which are key hypothesis .

The first step is to establish a discrete version of Theorem \ref{main result}, under a smallness assumption on the gradient. 
\begin{lemma}[Discrete regularity]\label{discrete reg}
    Let  $u \in C(B_1)$ be  a normalized viscosity solution to \eqref{main eq intr}, with $u(0) = 0$. There exist  $ \eta_{*} > 0$ and $0 < \rho < 1/2$ depending only on dimension, $ \lambda, \Lambda$, and $\beta$ such that: If  
    \begin{equation}\label{smallness k=1}
        \| \omega \|_{L^{\infty}(B_1)} + \|f\|_{L^{\infty}(B_1)}< \eta_{*}
    \end{equation}
    and 
    \begin{equation}\label{gradient smallnes k=1}
        |\nabla u(0) | < \eta_{*} 
    \end{equation}
    then there holds
        \begin{equation} 
         \sup\limits_{B_{\rho}} | u(x)  | \le \rho^{1 + \beta}. 
    \end{equation}
    
\end{lemma}

\begin{proof}
    For a $\delta_{*} > 0$ to be chosen, we apply lemma (\ref{Approx lemma}) to find $0 < \eta_{*} = \eta_{*} (\delta_{*})$ and  a function $h: B_{3/4} \to \mathbb{R}$ satisfying
    $$
        F(D^2 h) = 0 \quad \text{in} \quad B_{3/4}
    $$
    in the viscosity sense, $0 \in \mathcal{C}(h)$, and 
    $$
        \sup_{B_{1/2}} |v - h | < \delta_{*}
    $$
    for any function, $v$, satisfying \eqref{main eq intr approx} and \eqref{small regime}, with $\eta = \eta_{*}$. In view of the optimal $C_{loc}^{1,\alpha_0}$ interior regularity of $h$ ( see e.g., \cite{Caff, Teix1}), along with the fact that $ h(0) =  | \nabla h(0) |= 0  $, there exists a constant $C_0  = C_0(n, \lambda, \Lambda) > 0$ such that
$$
\sup_{B_r} |h(x)| \le \ C_0 r^{1 + \alpha_0}, \quad \forall 0 < r \le 1/2. 
$$
 Next, with $\beta$ fixed as in \eqref{sharp exponent beta}, we make the following universal choices:   
    \begin{equation}\label{choice of rho}
   0 < \rho \le \left( \frac{1}{2C_0} \right)^{\frac{1}{\alpha_0 - \beta}} \quad \text{and} \quad \delta_{*} := \frac{1}{2} \rho^{1 + \beta}. 
\end{equation}
To finish up, assuming that \eqref{smallness k=1} and \eqref{gradient smallnes k=1} hold true for the universal choices of $\eta_{*}$ and $\rho$ made above  (recall that the choice of $\delta_{*} $ determines $\eta_{*}$ through the compactness Lemma \eqref{Approx lemma})  we have, by triangular inequality, that 
\begin{eqnarray}\label{estimate for u}
    \sup\limits_{B_{\rho}} | u(x) | &\le&    \sup\limits_{B_{\rho}} | u(x) - h(x) | +  \sup\limits_{B_{\rho}} \left| h(x) \right| \nonumber \\
    &\le& \delta_{*} + C_0\rho^{1 +\alpha_0} \nonumber \\
    &\le& \frac{1}{2} \rho^{1 + \beta} + \frac{1}{2} \rho^{1 + \beta}.
\end{eqnarray}
 In conclusion, we have established
    \begin{equation}\label{iteration k=1 } 
         \sup\limits_{B_{\rho}} | u(x)  | \le \rho^{1 + \beta}. 
    \end{equation}
\end{proof}

Next, we obtain the main result of this subsection.

\begin{proposition}\label{Reg for small gradient}
    Let $u \in C(B_1)$ be a normalized viscosity solution to \eqref{main eq intr}, with $u(0) = 0.$ There exists small positive universal numbers $\rho$ and $\eta_{*}$ such that if
$$
    |\nabla u(0) | \le \eta_{*} t^{\beta}
$$
for some $0 < t \le \rho$, then  
\begin{equation}
    \sup\limits_{B_{t}} | u(x) | \le C t^{(1 + \beta)}
\end{equation}
for a universal constant $C = C(\rho, \beta) > 0.$
\end{proposition} 
\begin{proof}
    Initially we note that, by scaling and normalization,  we may assume that $\| \omega_F\|_{\infty} + \| f\|_{\infty} \le \eta_{*}$, where $\eta_{*}$ is the universal constant from Lemma \ref{discrete reg}. Indeed, for constants $\kappa  > 0$, and $ r \in (0,1)$ chosen so that 
    $$
         \| \omega_{F}\|_{L^{\infty} (B_r)} \le \eta_{*} /2 
    $$
    and 
    $$
        \kappa:= \frac{\eta_{*}/2}{ \| u \|_{L^{\infty}(B_1)} + r^{\frac{2 + \gamma}{1 + \gamma}} \|f\|_{L^{\infty}(B_1)}^{\frac{1}{1 + \gamma}}},
    $$
    we define the function $\tilde{u} (x) = \kappa u(r x)$, $x \in B_1$. Then $\tilde{u}$ solves 
    $$
    \tilde{\mathcal{D}}(x, \nabla \tilde{u}) \tilde{F} (x, D^2 \tilde{u}) = \tilde{f}(x)
 $$
 where 
 $$
    \tilde{F} (x, M) = \kappa r^2 F(rx , \kappa^{-1} r^{-2} M ),
$$ 
is a uniformly elliptic operator with same ellipticity constants as $F$. 
$$
    \tilde{\mathcal{D}}(x, \vec{b} ) = (\kappa r)^{\gamma} \mathcal{D}(rx, (\kappa r)^{-1}\vec{b} )
$$ satisfies the degeneracy condition \eqref{degener cond}, with the same constants. And 
$$
    \tilde{f}(x) = r^{2 + \gamma} \kappa^{1 + \gamma} f(rx).
$$ 
It is easily checked now that $\tilde{u}(0) = 0$, $| \tilde{u} | \le 1$, and 
$$ 
    \left( \| \omega_{\tilde{F}} \|_{L^{\infty}(B_1)} + \| \tilde{f} \|_{L^{\infty}(B_1)} \right) \le \eta_{*}. 
$$ 
The strategy of the proof is to iterate Lemma \ref{discrete reg}. That is, we want to show that if
\begin{equation}\label{grad iteration}
    |\nabla u(0) | \le \eta_{*} \rho_{*}^{k\beta}
\end{equation}
for some $k \in \mathbb{N}$, then there holds
\begin{equation}\label{iteration}
    \sup\limits_{B_{\rho^{k}}} | u(x) | \le \rho^{k(1 + \beta)}
\end{equation}
where $\rho$ is the radius granted also from lemma \ref{discrete reg}. We will prove that for each $k\in \mathbb{N}$, \eqref{grad iteration} implies \eqref{iteration} using an induction process. The step $k = 1$ follows directly from Lemma \ref{discrete reg}. We now proceed to the induction step. 

Assume that \eqref{grad iteration} implies \eqref{iteration} for $1, \dots, k$. Then, let $v_k : B_{1} \to \mathbb{R}$ be defined as
    $$
        v_k(x) : = \frac{u(\rho^{k} x) }{\rho^{k(1 + \beta)}} .
    $$
    Thus $v_k$ is a normalized solution of 
    $$
        \mathcal{D}_k (x, \nabla v_k) F_k(x, D^2 v_k)   = f_k(x)
    $$
    where
    $$
      F_k(z,N) = \rho^{ k(1- \beta)} F( \rho^{k} z, \rho^{-k(1- \beta)}N  )
$$
is another $(\lambda, \Lambda)$-uniformly elliptic operator,

$$
     \mathcal{D}_k ( z, \vec{\xi}) =  \rho^{- k\beta\gamma} \mathcal{D}( \rho^{k} z, \rho^{k\beta } \vec{\xi}  ) ,
$$
and, 
$$
    f_k (x) = \rho^{ k( 1 - \beta(1 + \gamma) )} f(\rho^{k} x) .
$$
Thus, using \eqref{contin of coeff}, we estimate 
$$
     \|\omega_{F_{k}} \|_{L^{\infty}(B_1)} \le K_1\rho^{\bar{\alpha} k}
$$
and using \eqref{holder cont of f},
$$
     \|f_k(x)\|_{L^{\infty}(B_1)} \le K_2 \rho^{ k[1 + \theta - \beta(1 + \gamma)]} \le K_2 \rho
$$
since the choice of $\beta$  yields $ 1 + \theta - \beta(1 + \gamma) \ge 0$.

Next, since  by hypotheses of induction, we have  $|\nabla u(0)| \le \eta_{*} \rho_{*}^{k\beta}$, then 
$$
    |\nabla v_k( 0 ) | \le \eta_{*}. 
$$
Therefore, $v_{k}$ is entitle to Lemma \eqref{discrete reg}, and  hence we  obtain 
\begin{eqnarray}
    \sup\limits_{B_{\rho}} | v_k(x)  | \le \rho^{(1 + \beta)} .
\end{eqnarray}
Scaling back to $u$ we have, 
$$
    \sup\limits_{B_{\rho^{k+1}}} | u(x)  | \le \rho^{(k+1)(1 + \beta)} .
$$
Thus, the induction step is proven. Next, to finish up, given $0 < t \le \rho$, let $k \in \mathbb{N}$ be so that $\rho^{k+1} < t \le \rho^k$. Thus, if 
$$
    |\nabla u(0)| \le \eta_{*} t^{\beta},
$$
then $|\nabla u (0)| \le \eta_{*}\rho^{k \beta}$, and hence, 
\begin{eqnarray}
    \sup\limits_{B_t} | u(x)| &\le& \sup\limits_{B_{\rho^k}} | u(x)| \nonumber \\
    & \le& \rho^{k(1 + \beta)} \nonumber \\
    &\le& \left( \rho^{-(1+\beta)} \right) t^{1 + \beta}.
\end{eqnarray}
\end{proof}

\subsection{Sharp regularity} In this subsection, we conclude the proof of Theorem \ref{main result}. First, by translation and normalization, we may assume $u(0) = 0$ and $\| u\|_{L^{\infty}(B_1)} \le 1$. 

Let $\eta_*, \rho > 0$ be the numbers granted by proposition \ref{Reg for small gradient}.  Following the dichotomy strategy, we  consider several cases:

\textbf{Case 1:} $|\nabla u(0)| \le  \eta_{*} \rho^{2\beta}$. 

In this case, we define the number 
\begin{eqnarray}\label{gradient parameter}
    \iota:= \left( \frac{| \nabla u(0) |}{\eta_{*}} \right)^{1/ \beta}.
\end{eqnarray}
Given   $t \le \rho^2$, there exists $k \in \mathbb{N}$ such that $\rho^{k+2} < t \le \rho^{k+1}$. We consider the sub-cases:

\textbf{Case 1.i)} $ \iota \le  t \le \rho^2$. Then, by \eqref{gradient parameter}, we have the following control: 
$$
    |\nabla u(0) | \le \eta_{*} t^{\beta} \le \eta \rho^{k\beta} .
$$
Thus, by Proposition \ref{Reg for small gradient}, 
\begin{eqnarray}
    \sup\limits_{B_t} |u(x) |  &\le& C t^{1 +\beta}
\end{eqnarray}
with $C  > 0$ universal.  Hence, 
\begin{eqnarray}
    \sup\limits_{B_t} |u(x) - \nabla u(0) \cdot x| &\le& \sup\limits_{B_t} |u(x) |  + |\nabla u(0)|  t \nonumber \\ 
    & \le& (C + \eta_{*}) t^{1+ \beta}. \nonumber
\end{eqnarray}
That is, $u$ is of class $C^{1, \beta}$ at $0$. 

\textbf{Case 1.ii)} $0 < t < \iota \le \rho^2$. In this case we consider the scaled function 
$$
    v_{\iota} (x): = \frac{u(\iota x)}{\iota^{1 + \beta}} , \quad x \in B_1 .
$$
It follows from \eqref{gradient parameter} that 
$$
    |\nabla u(0) | = \eta_{*} \iota^{\beta}. 
$$ 
and by proposition \ref{Reg for small gradient},
$$
    \sup\limits_{B_1} |v_{\iota} (x)| = \sup\limits_{B_{\iota}} \left| \frac{u(\iota x)}{\iota^{1 + \beta}} \right| \le C.
$$
Therefore, $v_{\iota}$ is a bounded viscosity solution to 
\begin{equation}\label{kappa equation}
    \tilde{\mathcal{D}}(x, \nabla v_{\iota} ) \tilde{F}( x, D^2 v_{\iota} ) = f_{\iota}(x)
\end{equation}
where $\tilde{F}$ is another $(\lambda, \Lambda)$-elliptic operator. As before, $\tilde{\mathcal{D}}(x, \vec{\xi}) $ still satisfies the degeneracy condition \eqref{degener cond}, and $f_{\iota} (x) = \iota^{1 - \beta + \gamma \beta} f(\iota x)$. Note that in view of \eqref{holder cont of f}, 
$$
    \| f_{\iota} (x)\|_{L^{\infty}(B_1)} \le K_2 \iota^{1 + \theta - \beta(1 + \gamma)} \le K_2.
$$
That is, the right hand side of \eqref{kappa equation} is universally bounded. Therefore, from the $C^{1,+}$ interior  regularity available for $v_{\iota}$, and the fact the  $|\nabla v_{\iota} (0) | = \eta_{*} > 0$, the we can find a universal small radius $r > 0$ (independent of $\iota$)  such that 
$$
    \frac{\eta_{*}}{2} \le |\nabla v_{\iota} (x)| \le 2 \eta_{*}, \quad \text{in} \quad B_r .
$$
Hence, $v_{\iota}$ satisfies the following uniformly elliptic equation with bounded right-hand side
$$
    \tilde{F} (x, D^2 v_{\iota} ) = \tilde{\mathcal{D}}(x, \nabla v_{\iota} )^{-1} f_{\iota}(x), \quad \text{in} \quad B_{r}.
$$
Therefore, since by the regularity results established in \cite{Teix1}, solutions to the above equation are almost as regular as solutions to a constant coefficient, homogeneous equation, i.e., they are $C_{loc}^{1,\alpha_0^{-}}$. In particular,  we obtain 
$$
    \sup\limits_{B_{\tau}} |v_{\iota} (x) - \nabla v_{\iota} (0) \cdot x | \le C \tau^{1 + \beta}
$$
for all $0 < \tau \le \frac{r}{2}.$ Hence, 
\begin{equation}\label{almost there equation}
    \sup\limits_{B_t} | u(x) - \nabla u(0) \cdot x| \le C t^{1 + \beta}
\end{equation}
for all $0 < t \le \iota \frac{r}{2}$. 

We now extend \eqref{almost there equation} to the interval $\displaystyle \left( \iota \frac{r}{2} , \iota \right)$.  For $t \in \left( \iota \frac{r}{2} , \iota \right)$, then by case (1.i) with $t= \iota$, we have 
\begin{eqnarray}
    \sup\limits_{B_t} | u(x) - \nabla u(0) \cdot x| &\le & \sup\limits_{B_{\iota}} | u(x) - \nabla u(0) \cdot x| \nonumber \\
    &\le& C \iota^{1 + \beta} \nonumber \\
    &\le& C \left( \frac{2}{r} \right)^{1 + \beta} t^{1 + \beta},
\end{eqnarray}
and this concludes the proof for case 1. 

\textbf{Case 2:} $|\nabla u(0)| > \eta_{*} \rho^{2\beta}$. In this case, we consider the function
$$
    \tilde{u} (x) = \frac{\eta_{*}\rho^{2\beta}}{|\nabla u(0)|} u(x), 
$$
and then are back to case 1. 
 \section{Hessian continuity at local extrema}\label{sct hessian cont}
 \subsection{Flatness estimate} \label{hessian appr}

Once more, example \eqref{example of hessian cont} enlightens us with the idea that, at certain points,  we can achieve even higher regularity. Indeed, in that example, the origin is a minimum point of the solution, at which the source function vanishes. In this section, we formally prove this observation.  That is, we now show that at a local extrema point solutions are $C^{2,\varepsilon}$ provided that,  at that point, the source function vanishes at a bigger rate than ellipticity degenerates. 

The first step is to prove a flatness improvement result at the local extrema. However, unlike Lemma \ref{Approx lemma}, which establishes the proximity between functions within the space of solutions, the subsequent lemma asserts that solutions exhibit nearly constant behavior near an extremum point.

\begin{lemma}[Flatness estimate]\label{Flatness estimate}
 Let $\varphi \in C(B_1)$ be a normalized viscosity solution to 
\begin{equation}
\mathcal{D}(x, \nabla \varphi ) F(x, D^2 \varphi ) =  f(x) \quad \text{in} \quad B_1,
\end{equation}
for some $(\lambda, \Lambda)$-elliptic operator $F:B_1 \times Sym(n) \to \mathbb{R}^n$. Assume, that $x_0 \in B_{1/2}$ is a local minimum. Then, given $\eta > 0$ and  there exists $\rho (\eta) > 0$ depending only upon $\eta, n, \lambda, \Lambda$,  such that if   
$$
   \| \omega_F\|_{\infty} +  \| f \|_{\infty} \le \delta,
$$
 holds for $0 < \delta \le \rho(\eta)$, then 
\begin{equation}
    \sup\limits_{B_{1/8}(x_0)} \left( \varphi - \varphi(x_0) \right) \le \eta .
\end{equation}
\end{lemma}

\begin{proof}
Suppose, by contradiction, that there exists $\eta_0 > 0$ and  sequences of functions $\{ \varphi_k\},$  a sequence of $(\lambda, \Lambda)$-elliptic operators $F_k : B_1 \times Sym(n) \to \mathbb{R}$, a sequence $x_k \in B_{1/2}$ of local minimum of $\varphi_k$, and a linked sequence of functions $\{f_k\},$ such that 

    \begin{equation}\tag{i}
         -1 \le \varphi_k \le 1 
    \end{equation}

     \begin{equation}\tag{ii}
         \mathcal{D}_k (x, \nabla \varphi_k ) F_{k} (x, D^2 \varphi_k ) = f_k(x)  \quad \text{in} \quad  B_1
     \end{equation}
 in the viscosity sense. And 
 \begin{equation}\tag{iii}
      \| \omega_{F_k}\|_{\infty} +  \| f_k \|_{\infty} \le \frac{1}{k} .
 \end{equation}
 However, 
    \begin{equation}\label{flat contrad eq}
             \sup\limits_{B_{1/8}(x_0)} \left( \varphi_k - \varphi_{k}(x_k) \right) \ge \eta_0 > 0, 
    \end{equation} 
   
Next, arguing as in Lemma \ref{Approx lemma}, the sequence $\{ \varphi_k \}_{k \in \mathbb{N} }$ is pre-compact in $ C_{loc}^{1,+} (B_1)$. Thus, passing to a subsequence if necessary, $\varphi_k $ converges locally uniformly in $B_1$ to a continuous function  $\varphi_{\infty}$ in the $C^{1, +}$ topology.  
By uniform ellipticity and (iii),  $F_k (x, M) \to F_{\infty}(M)$ locally uniformly in  $B_1 \times Sym(n)$ to some constant coefficient, $(\lambda, \Lambda)$-operator. And again,  using the degeneracy condition \eqref{degener cond}, we have 
\begin{equation}\label{seq ineq}
    \left\{\begin{matrix}
\lambda | \nabla \varphi_k |^{\gamma} | F_k( x, D^2 \varphi_k ) | \le  |f_{k} (x)| \\ 
        \\
 \Lambda | \nabla \varphi_k  |^{\gamma} | F_k( x, D^2 \varphi_k ) | \ge  |f_{k} (x)| .
\end{matrix}\right.
\end{equation}
 By stability of viscosity solutions, passing the limit as $k \to \infty$ in \eqref{seq ineq} we see that $\varphi_{\infty}$ is a normalized solution to 
\begin{equation}\label{limit eq}
    |\nabla \varphi_{\infty} |^{\gamma} F_{\infty} (D^2 \varphi_{\infty} )  =0
\end{equation}
    in the viscosity sense, which in view of \cite[Lemma 6]{IS} implies
    \begin{equation}\label{equa after cancelation}
        F_{\infty} ( D^2 \varphi_{\infty} ) = 0 \quad \text{in} \quad B_{2/3} ,
    \end{equation}
    in the viscosity sense.  Furthermore, $x_k \to x_{\infty}$ and, by the uniform convergence, $x_{\infty}$ is a local minimum of $\varphi_{\infty}$. Therefore, by strong maximum principle, we have $\varphi_{\infty} \equiv \text{constant}$,  and hence we get a contradiction with \eqref{flat contrad eq} for $k$ sufficiently large. 
\end{proof}

 \begin{remark}
     If $\mathcal{D}(x, p) = |p|^{\gamma}$, then no assumption on the coefficients of $F$ is needed. Indeed, solutions to 
     $$
        |\nabla \varphi|^{\gamma} F(x, D^2 \varphi ) = f(x)
     $$
     are solution to 
     $$
         |\nabla \varphi|^{\gamma} \mathcal{M}_{\lambda, \Lambda}^{-} (D^2 \varphi) \le |f| \quad \text{and} \quad  |\nabla \varphi|^{\gamma}\mathcal{M}_{\lambda, \Lambda}^{+} (D^2 \varphi ) \ge - |f|.
     $$
     Therefore, stability of viscosity solutions allow us the same conclusion as in \eqref{equa after cancelation}. 
 \end{remark}
\subsection{Higher regularity}\label{itera flat}

We want to show that $u$ is  $C^{2, \varepsilon}$ at the origin, for the sharp $\varepsilon$ defined as in \eqref{varepsilon}.  Recall that now we are assuming that $0$ is a local minimum of $u$ and 
$$
   | u| \le 1 \quad \text{and} \quad |f(x)| \le K_2 |x|^{\theta} ,
$$ which are  key hypothesis of Theorem \ref{hessian cont}. 

\begin{proof}[Proof of Theorem \ref{hessian cont}]
As before, there is no loss of generality in assuming that $u(0) = 0$. For $ 0 < \rho_{*} \le 1/8$ to be chosen later, consider the scaled,  normalized function $v(x) = u( \rho_{*} x) $ defined in $B_1$. It is easily checked that $v$ satisfies
    $$
        \mathcal{D}_{\rho} (x, \nabla v) F_{\rho} ( x, D^2 v) = f_{\rho}(x) \quad \text{in}  \quad  B_1
    $$
where $F_{\rho} (x, M) = \rho^2 F(\rho x, \rho^{-2} M )$, is a $(\lambda, \Lambda)$- uniformly elliptic operator,  $\mathcal{D}_{\rho} (x, \vec{\xi} ) = \rho^{\gamma} \mathcal{D}(\rho x, \rho^{-1} \vec{\xi} )$ still satisfy the degeneracy condition \eqref{degener cond} with the same constants, and finally 
$$
    f_{\rho}(x)=\rho^{2 + \gamma} f(\rho x).
$$ Hence, using \eqref{contin of coeff} and \eqref{holder cont of f}, we have 
$$
\|\omega_{F_{\rho} }\|_{L^{\infty}(B_1)} \le K_1 \rho^{\bar{\alpha} } \quad \text{and} \quad \|f_{\rho} \|_{L^{\infty}(B_1)} \le  K_2\rho^{(2 + \gamma + \theta)} . 
$$
Next, we select $\eta = 8^{-(2 + \varepsilon)}$ and let  $\rho(\eta)$ be the correspondent smallness regime provided by Lemma \ref{Approx lemma}. With this $\rho(\eta) $ in hands, we choose 
\begin{equation}\label{small regime loc min}
    0 < \rho_{*} \le \min \left\{ \left( \frac{\rho(\eta)}{K_1} \right)^{\frac{1}{\bar{\alpha}}} , \left( \frac{\rho(\eta)}{K_2} \right)^{\frac{1}{2 + \gamma + \theta}} \right\}. 
\end{equation}
With these choices, $v$ is under the assumptions of Lemma \ref{Flatness estimate}, for $\eta = 8^{-(2 + \varepsilon)}$. Thus, 
$$
    \sup\limits_{B_{1/8}}  v(x)  \le 8^{-(2 + \varepsilon)} . 
$$
In the sequel, let us define $v_2: B_1 \to \mathbb{R}$ by 
$$
    v_2(x) = 8^{2 + \varepsilon}  v \left(\frac{x}{8} \right)   .
$$
Then $0$ is still a local minimum for $v_2$, and one readly verifies that 
$$
    | v_2 | \le 1  
$$
and 
$$
|\nabla v_2 |^{\gamma} \tilde{F}_{\rho} (x, D^2 v_2)= \tilde{f}_{\rho}(x), 
$$
where $ \tilde{F}_{\rho} (x, M) = 8^{\varepsilon}F_{\rho}( 8^{-1}x, 8^{-\varepsilon} M ) $, $\tilde{\mathcal{D}}_{\rho} (x, \vec{\xi} ) = 8^{(1+\varepsilon)\gamma} \mathcal{D}_{\rho}(\rho x, 8^{-(1 +\varepsilon)} \vec{\xi} )$ ,  and 
$$
      \tilde{f}_{\rho}(x)=8^{\varepsilon + (1 +\varepsilon)\gamma} \rho_{*}^{2 +\gamma} f_{\rho} \left( \frac{\rho_{*}x}{8} \right) .
$$
Because of the sharp choice of $\varepsilon$ we have 
$$
   \varepsilon + (1 +\varepsilon)\gamma = \theta  
$$
and hence, using conditions \eqref{contin of coeff} and  \eqref{holder cont of f}, we still obtain 
$$
   \|\tilde{\beta}_{F_{\rho}}\|_{\infty} \le K_1 \rho_{*}^{\bar{\alpha}} \quad \text{and} \quad \| \tilde{f}_{\rho} \|_{\infty} \le K_2 \rho_{*}^{2 + \gamma + \theta} . 
$$
Indeed, for any $x \in B_1$ and $M \in Sym(n)$ we have 
 \begin{eqnarray}
    \left| \frac{8^{\varepsilon} \left( F_{\rho} (8^{-1} x, 8^{-\varepsilon}M) - F_{\rho}(0, 8^{-\varepsilon} M) \right)}{1 + \|M \| } \right| &\le& \left| \frac{8^{\varepsilon} \rho^2 \left(  F (8^{-1} \rho x, 8^{-\varepsilon}\rho^{-2} M) -  F(0, 8^{-\varepsilon} \rho^{-2} M) \right)}{1 + \|M \| } \right| \nonumber \\
    &\le& K_1 \frac{8^{\varepsilon} \rho^2 (1 + 8^{-\varepsilon}\rho^{-2} \|M\|)\,  |8^{-1} x|^{\bar{\alpha}}}{ 1 + \| M\| } \nonumber \\
    &\le& K_1 \frac{(8^{\varepsilon} \rho^2  +  \|M\| )\,  |8^{-1} x|^{\bar{\alpha}}}{ 1 + \| M\| } \nonumber \\
    &\le& \rho_{*}^{\bar{\alpha}} K_1 |x|^{\bar{\alpha}},
\end{eqnarray}
since $\rho_{*} \le 8^{-1}$  yields $8^{\varepsilon} \rho_{*}^{2} \le 8^{\varepsilon - 2 } \le 1.$ And,  
\begin{eqnarray}
     | \tilde{f}_{\rho}(x)| &=& \left| 8^{\varepsilon + (1 +\varepsilon)\gamma} \rho_{*}^{2 +\gamma} f_{\rho} \left( \frac{\rho_{*}x}{8} \right) \right|  \nonumber \\
     &\le& 8^{\varepsilon + (1 +\varepsilon)\gamma} \rho_{*}^{2 +\gamma} K_2 \left( \frac{\rho_{*}^{2}}{8} \right)^{^{\theta}} \nonumber \\
     &\le&  K_2 8^{\varepsilon + (1 +\varepsilon)\gamma - \theta} \rho_{*}^{2 + \gamma + \theta} \nonumber \\
     &=& K_2 \rho_{*}^{2 + \gamma + \theta} . 
\end{eqnarray}
That is, $v_2$ too is under the hypothesis of Lemma \ref{Flatness estimate}. Hence, 
\begin{equation}\label{iter step 2}
    \sup\limits_{B_{1/8}} v_2 (x) \le 8^{-(2 + \varepsilon)} . 
\end{equation}
Re-scaling \eqref{iter step 2} back to $v$ yields
\begin{equation}
    \sup\limits_{B_{1/64}} v(x) \le 8^{-2(2 + \varepsilon)} .
\end{equation}
Iterating inductively the above reasoning gives the following geometric decay:
\begin{equation}
    \sup\limits_{B_{8^{-k}}} v(x) \le 8^{-k(2 + \varepsilon)} .
\end{equation} 
Finally, given $0 < r \le \frac{\rho_{*}}{8}$, there exists $ k \in \mathbb{N}$, such that $8^{-(k+1)} < \frac{r}{\rho_{*}} \le 8^{-k}$,  therefore, 
\begin{eqnarray}
    \sup\limits_{B_r} u(x) &\le& \sup\limits_{B_{\frac{r}{\rho_{*}} }} v(x) \nonumber \\ 
    &\le& \sup\limits_{B_{8^{-k}}} v(x) \nonumber \\
    &\le& 8^{-k(2 + \varepsilon)} \nonumber \\
    &\le& \left( \frac{8}{\rho_{*}} \right)^{(2 + \varepsilon)} r^{2 + \varepsilon} \nonumber \\
    &= C r^{2 + \varepsilon} ,
\end{eqnarray}
where $C = C(n, \lambda, \Lambda, K_1, K_2, \bar{\alpha}, \gamma, \theta )$. Thus, $u$ is $C^{2,\varepsilon}$ differentiable at $0$, with $Du(0) = D^2 u(0) = 0. $
\end{proof}


\noindent \textbf{Acknowledgement.}
The author is grateful to Eduardo Teixeira for his insightful comments which significantly enhanced the final presentation of this article.




\noindent

\bibliographystyle{amsplain, amsalpha}

\end{document}